\documentclass[10pt]{article}
\usepackage{amsfonts,amsmath,amsthm, amssymb}
\usepackage{latexsym, euscript, epic, eepic}
\newtheorem{theorem}{Theorem}[section]
\newtheorem{lemma}[theorem]{Lemma}%[section]
\newtheorem{sublemma}[theorem]{Sublemma}%[section]
\newtheorem{definition}{Definition}[section]
\newtheorem{proposition}[theorem]{Proposition}%[section]

\newtheorem{remark}{Remark}[section]

\newtheorem*{TheoremA}{Theorem A}
\newtheorem*{TheoremB}{Theorem B}
\newtheorem*{TheoremC}{Theorem C}
\newtheorem*{TheoremD}{Theorem D}
\newtheorem*{TheoremE}{Theorem E}
\newtheorem*{TheoremF}{Theorem F}

\def\diam{\mathop{\hbox{\rm diam}}}
\def\diams{\rm\small diam}

\def \d{\delta}
\def \e{\epsilon}

\def \M{\mathcal{M}}

\begin{document}

\title{Nonadditive  Measure-theoretic Pressure and  Applications to Dimensions  of an Ergodic Measure }
\author{     Yongluo Cao$^{\dag}$, Huyi Hu$^{\ddag}$ and Yun Zhao$^{\dag,\footnote{Corresponding author }}$ \\
\small\it $\dag$ Department of Mathematics, Soochow  University, Suzhou 215006, Jiangsu, P.R.China\\
  \small\it (email: ylcao@suda.edu.cn, zhaoyun@suda.edu.cn)\\
  \small\it $\ddag$ Department of Mathematics, Michigan State University, East Lansing, MI 48824,
  USA\\
  \small\it (email: hu@math.msu.edu)
  }
\date{}
 \footnotetext{2000 {\it Mathematics Subject classification}:
 37D35, 37A35, 37C45}
 \maketitle

\begin{center}
\begin{minipage}{120mm}
{\small {\bf Abstract.} Without any additional conditions on
subadditive potentials, this paper defines subadditive
measure-theoretic pressure, and shows  that the subadditive
measure-theoretic pressure for ergodic measures can be described
in terms of measure-theoretic entropy and a constant associated
with the ergodic measure. Based on the definition of topological
pressure on non-compact set, we give another equivalent definition
of subadditive measure-theoretic pressure, and obtain an inverse
variational principle. This paper also studies  the supadditive
measure-theoretic pressure which has similar formalism as the
subadditive measure-theoretic pressure. As an application of the
main results, we prove that an average conformal repeller admits
an ergodic measure of maximal Haudorff dimension. Furthermore, for
each ergodic measure supported on an average conformal repeller,
we construct a set whose dimension is equal to the dimension of
the measure. }
\end{minipage}
\end{center}

\vskip0.5cm

{\small{\bf Key words and phrases.} \  nonadditive,
measure-theoretic pressure, variational principle, ergodic
measure, Hausdorff dimension.}\vskip0.5cm

%%%%%%%%%%%%%%%%%%%%%%%%%%%%%%%%%%%%%%%%%%%%%%%%
%%%%%%%%%%%%%%%%%%%%%%%%%%%%%%%%%%%%%%%%%%%%%%%%
\section{Introduction.}
\setcounter{equation}{0}
%%%%%%%%%%%%%%%%%%%%%%%%%%%%%%%%%%%%%%%%%%%%%%%%
%%%%%%%%%%%%%%%%%%%%%%%%%%%%%%%%%%%%%%%%%%%%%%%%

It is well-known that the topological pressure for additive potentials
was first introduced by Ruelle  for expansive maps acting on
compact metric spaces (\cite{rue}), furthermore he formulated a
variational principle for the topological pressure in that paper.
Later, Walters (\cite{wal}) generalized these results to general
continuous maps on compact metric spaces. In \cite{pes2}, Pesin
and Pitskel' defined the topological pressure for noncompact sets
which is a generalization of Bowen's definition of topological
entropy for noncompact sets (\cite{bo1}), and  they proved the
variational principle under some supplementary conditions.
The notions of the topological pressure, variational principle and
equilibrium states play a fundamental role in statistical
mechanics, ergodic theory and dynamical systems (see the books
\cite{bo3,w2}).

Since the work of Bowen (\cite{bo2}), topological pressure becomes
a fundamental tool for study of the dimension theory in conformal
dynamical systems (see \cite{pes1}).  Different versions of
topological pressure were defined in dimension theory and ergodic
theory.  It has been found out that the dimension of nonconformal
repellers can be well estimated by the zero of pressure function
(see e.g. \cite{ban,ba,cli,falconer,zhang}).  Falconer
(\cite{falconer}) considered the thermodynamic formalism for
subadditive potentials on mixing repellers. He proved the
variational principle under some Lipschitz conditions and bounded
distortion assumptions on subadditive potentials. Barreira
(\cite{ba}) defined topological pressure for arbitrary sequences
of potentials on arbitrary subsets of compact metric spaces, and
proved the variational principle under some convergence
assumptions on the potentials.  However, the conditions given by
Falconer and Barreira are not usually satisfied by general
subadditive potentials.

In \cite{c}, the authors generalized the results of Ruelle and
Walters to subadditive potentials in general compact dynamical
systems. They defined the subadditive topological pressure and
gave a variational principle for the subadditive topological
pressure.  We mention that their result do not need any additional
conditions on either the subadditive potentials or the spaces, as
long as they are compact metric spaces.  In \cite{z1}, the authors
defined the subadditive measure-theoretic pressure by using
spanning sets, and obtained a formalism similar to that for
additive measure-theoretic pressure in \cite{h} under tempered
variation assumptions on subadditive potentials. Another
equivalent definition of subadditive measure-theoretic pressure is
given in \cite{z2} under the same conditions on the sub-additive
potentials by using the definition of topological pressure on
noncompact sets. In \cite{gh}, Zhang studied the local
measure-theoretic pressures for subadditive potentials. The
pressure is local in the sense that an open cover is fixed.

Part of this paper is a continuation of the work in \cite{z1} and
\cite{z2}. We modify the definition of subadditive
measure-theoretic pressure there, and remove the extra condition
for the formulas. More precisely, for an ergodic measure $\mu$,
the subadditive measure-theoretic pressure defined by using
Carath\'eodory structure can be expressed as the sum of the
measure-theoretic entropy and the integral of the limit of the
average value of the subadditive potentials. Consequently, this
paper gives equivalence to an alternative definition
of subadditive measure-theoretic pressure by considering
spanning sets about which the Bowen balls only cover a set
of measure greater than or equal to $1-\delta$.
The results we get here do not need any
additional assumptions on the subadditive sequences and the
topological dynamical systems, except for compactness of the
spaces. Meanwhile, we also define the measure-theoretic pressure
for supadditive potentials in a similar way and obtain the same
properties.

The present work is also motivated by the dimension theory in
dynamical systems.
Dimensions of a compact invariant set can often be determined
or estimated by a unique root of certain pressure functions.
As an application, we proved that the dimension of an
average conformal repeller which is introduced in \cite{ban}
satisfies a variation principle, i.e., the dimension of this
repeller is equal to the dimension of some ergodic measure
supported on it. Moreover, for each ergodic measure supported on
an average conformal repeller, we could construct a certain set
with the same dimension of the measure.

The paper is organized in the following manner. The main results,
as well as definitions of the measure-theoretic pressure and lower
and upper capacity measure-theoretic pressure for subaddtive potentials,
are given in Section 1. We prove Theorem~A and B, the results
related to subadditive potentials, in Section 2.  Section 3 is
devoted to supadditive potentials, where we give definitions of the
pressures, and prove Theorem~C and D. Section 4 is for application
to dimension of average conformal repellers, where the results are
stated in Theorem~E and F.

%%%%%%%%%%%%%%%%%%%%%%%%%%%%%%%%%%%%%%%%%%%%%%%%
%%%%%%%%%%%%%%%%%%%%%%%%%%%%%%%%%%%%%%%%%%%%%%%%
\section{Main results}
\setcounter{equation}{0}
%%%%%%%%%%%%%%%%%%%%%%%%%%%%%%%%%%%%%%%%%%%%%%%%
%%%%%%%%%%%%%%%%%%%%%%%%%%%%%%%%%%%%%%%%%%%%%%%%

Let $(X,T)$ be a topological dynamical systems(TDS),
that is, $X$ is a compact metric space with a metric $d$,
and $T : X\rightarrow X$ is a continuous transformation.
Denote by $\mathcal{M}_T$ and $\mathcal{E}_T$ the set of all
$T$-invariant Borel probability measures on $X$ and the set of
ergodic measures respectively.
For each $\mu\in \mathcal{M}_T$,
let $h_{\mu}(T)$ denote the measure-theoretic entropy of $T$ with
respect to $\mu$.

A sequence ${\mathcal{F}}=\{ f_n\}_{n\geq 1}$ of continuous
functions on $X$ is a \emph{subadditive potential} on $X$, if
\[
f_{n+m}(x)\leq f_n(x) +f_m(T^n x)
\ \ \text{for all} \ \ x\in X, ~ n,m\in {\mathbb{N}}.
\]
For $\mu\in \mathcal{M}_T$, let ${\mathcal{F}}_{*}(\mu)$ denote
 the following limit
\[ {\mathcal{F}}_{*}(\mu)=\lim\limits_{n\rightarrow\infty}
 \frac{1}{n} \int f_n {\mathrm{d}}\mu
=\inf_{n\geq 1} \Bigl\{\frac{1}{n}\int f_n{\mathrm{d}} \mu\Bigr\}.
 \]
The limit exists since $\{\int f_n{\mathrm{d}} \mu\}_{n\geq 1}$ is
a subadditive sequence. Also, by subadditive ergodic theorem
\cite{king} the limit $\lim\limits_{n\rightarrow\infty}(1/n)f_n$
exists $\mu$-almost everywhere for any $\mu\in \M_T$.

Let $d_n(x,y)=\max\{d(T^{i}(x),T^{i}(y)):i=0,\cdots,n-1\}$ for any
$x,y\in X$, and $B_n(x,\epsilon)=\{y\in X:d_n(x,y)<\epsilon\}$.
A set $E\subseteq X$ is said to be an $(n,\epsilon)$-separated
subset of $X$ with respect to $T$ if $x,y\in E, x\neq y$,
implies $ d_{n}(x,y)> \epsilon$.
A set $F\subseteq X$ is said to be an $(n,\epsilon)$-spanning subset
of $X$ with respect to $T$ if $\forall x\in X$, $\exists y\in F$ with
$d_{n}(x,y)\leq \epsilon$.
For each $\mu\in \mathcal{M}_T$, $0<\delta <1$, $n\geq 1$ and
$\epsilon >0$, a subset $F\subseteq X$ is an
$(n,\epsilon,\delta)$-spanning set if the union $\bigcup_{x\in F}
B_n(x,\epsilon)$ has $\mu$-measure more than or equal to
$1-\delta$.

Recall that the \emph{subadditive topological pressure} of $T$ with respect to
a subadditive potential ${\mathcal{F}}=\{ f_n\}_{n\geq 1}$ is give by
\[
P(T,{\mathcal{F}})
=\lim\limits_{\epsilon\rightarrow 0}P(T,{\mathcal{F}},\epsilon),
\]
where
\begin{eqnarray*}
&&P(T,{\mathcal{F}},\epsilon)=\limsup\limits_{n\rightarrow\infty}
\frac{1}{n} \log P_{n}(T,{\mathcal{F}},\epsilon),\\
&&P_{n}(T,{\mathcal{F}},\epsilon)=\sup \{ \sum\limits_{x\in E}
e^{f_{n}(x)} : E\  \text{is  an}\ (n,\epsilon)\text{-separated
subset of}\ X \}.
\end{eqnarray*}
(See e.g. \cite{ba}, \cite{c}.)
It satisfies a variational principle (see \cite{c} for a proof and \cite{z3}
for its random version).

In \cite{kat}, Katok showed that measure-theoretic entropy can be regarded
as the growth rate of the minimal number of $\epsilon$-balls in the $d_n$
metric that cover a set of measure more than or equal to $1-\delta$.
Motivated by the observation, the following definition can be given.

\begin{definition}\label{defPmu}
Given a subadditive potential $\mathcal{F}=\{f_n\}$,
for $\mu\in \mathcal{E}_T$, $0< \delta < 1$, $n\geq 1$, and $\epsilon>0$,
put
 \begin{eqnarray*}
 &&P_{\mu}(T,{\mathcal{F}},n,\epsilon,\delta)=\inf \{ \sum\limits_{x\in
 F} \exp[\sup_{y\in B_n(x,\epsilon)} f_{n}(y) ]\mid F\ \hbox{is an}\ (n,\epsilon,\delta)- \hbox{spanning set} \},\\
&&P_{\mu}(T,{\mathcal{F}},\epsilon,\delta)=\limsup\limits_{n\rightarrow\infty}
 \frac{1}{n} \log  P_{\mu}(T,{\mathcal{F}},n,\epsilon,\delta),\\
&&P_{\mu}(T,{\mathcal{F}},\delta)=\liminf\limits_{\epsilon\rightarrow
 0} P_{\mu}(T,{\mathcal{F}},\epsilon,\delta),\\
 &&P_{\mu}(T,{\mathcal{F}})=\lim\limits_{\delta\rightarrow
 0} P_{\mu}(T,{\mathcal{F}},\delta).
 \end{eqnarray*}
$P_{\mu}(T,{\mathcal{F}})$ is said to be the
\emph{subadditive measure-theoretic pressure} of $T$ with respect to
$\mathcal{F}$.
\end{definition}

\begin{remark}\label{RmkPmu}
It is easy to see that $P_{\mu}(T,{\mathcal{F}},\delta)$ increases
with $\delta$.  So the limit in the last formula exists.
In fact, it is proved in \cite[Theorem 2.3]{czc} that
$P_{\mu}(T,{\mathcal{F}},\delta)$ is independent of $\delta$.
Hence, the limit of $\delta\rightarrow 0$ is redundant in the definition.
The same phenomenon can also be seen for measure-theoretic entropy
(see \cite[Theorem 1.1]{kat}).
\end{remark}

\begin{remark}\label{Rmkvarphi}
If $\mathcal{F}=\{ f_n \}$ is additive generated by a continuous
function $\varphi$,  that is,
$f_n(x)=\sum_{i=0}^{n-1} \varphi (T^i x)$ for some continuous
function $\varphi:X\rightarrow \mathbb{R}$,
then we simply write $P_{\mu}(T,{\mathcal{F}})$ as
$P_{\mu}(T,\varphi)$.
\end{remark}

An alternative definition of subadditive measure-theoretic
pressure can be given by using the theory of Carath\'eodory
structure (see \cite{pes1} for more details in additive case).

Let $Z\subseteq X$ be a subset of $X$, which does not have to be
compact nor $T$-invariant.
Fix $\epsilon >0$, we call $\Gamma=\{B_{n_i}(x_i,\epsilon)\}_i$
a \emph{cover of $Z$} if $Z\subseteq\bigcup_i B_{n_i}(x_i,\epsilon)$. For
$\Gamma=\{B_{n_i}(x_i,\epsilon)\}_i$, set
$n(\Gamma)=\min_i\{n_i\}$.

The theory of Carath\'eodory dimension characteristic ensures
the following definitions.

\begin{definition}\label{defPmu*}
Let $s\geq 0$, put
\begin{eqnarray}\label{other1}
M(Z,{\mathcal{F}},s,N,\epsilon)=\inf_\Gamma\sum_i
\exp\bigl(-sn_i+\sup_{y\in B_{n_i}(x_i,\epsilon) } f_{n_i}( y)\bigr),
\end{eqnarray}
where the infimum is taken over all covers $\Gamma$ of $Z$ with
$n(\Gamma)\geq N$.  Then let
%\begin{equation}\label{other2}
\begin{gather}%\label{other2}
m(Z,{\mathcal{F}},s,\epsilon)
=\lim_{N\rightarrow\infty}M(Z,{\mathcal{F}},s,N,\epsilon),  \label{other2}\\
P_{Z}(T,{\mathcal{F}},\epsilon)
=\inf \{ s: m(Z,{\mathcal{F}},s,\epsilon)=0 \}
 =\sup \{ s: m(Z,{\mathcal{F}},s,\epsilon)=+\infty \} \label{other3}, \\
P_{Z}(T,{\mathcal{F}})=\liminf_{\epsilon\rightarrow \label{other4}
0}P_{Z}(T,{\mathcal{F}},\epsilon),
\end{gather}
where $P_{Z}(T,{\mathcal{F}})$ is called a \emph{subadditive
topological pressure} of $T$ on the set $Z$ (w.r.t. ${\mathcal{F}}$).

Further, for $\mu \in \mathcal{M}_T$, put
\begin{gather}
P_{\mu}^{*}(T,{\mathcal{F}},\epsilon)
=\inf \{ P_{Z}(T,{\mathcal{F}},\epsilon):\mu (Z)=1\}, \notag \\
P_{\mu}^{*}(T,{\mathcal{F}})=\liminf_{\epsilon\rightarrow
0}P_{\mu}^{*}(T,{\mathcal{F}},\epsilon), \label{pressure1}
\end{gather}
where $P_{\mu}^{*}(T,{\mathcal{F}})$ is called a \emph{subadditive
measure-theoretic pressure} of $T$ with respect to $\mu$.
\end{definition}

It is easy to see that the definition is consistent with that
given in \cite{ba} by using arbitrary open covers.

Lower and upper capacity topological pressure for additive sequence
were defined in \cite{pes1}.  Now we give similar definitions:

\begin{definition}\label{defCP}
Put
\[
\Lambda(Z,{\mathcal{F}},N,\epsilon)
=\inf_\Gamma\sum_i\exp\bigl(\sup_{y\in B_{N}(x_i,\epsilon) }f_N(y)\bigr),
\]
where the infimum is taken over all covers  $\Gamma$ of $Z$ with
$n_i= N$ for all $i$.  Then we set
\begin{gather}
\underline{CP}_{Z}(T,{\mathcal{F}},\epsilon)
=\liminf_{N\rightarrow\infty} \frac{1}{N}
    \log\Lambda(Z,{\mathcal{F}},N,\epsilon),      \label{sup1}  \\
\overline{CP}_{Z}(T,{\mathcal{F}},\epsilon)
=\limsup_{N\rightarrow\infty} \frac{1}{N}\log
\Lambda(Z,{\mathcal{F}},N,\epsilon).  \label{sup2}
\end{gather}

For $\mu \in \mathcal{M}_T$, define
\begin{gather*}
\underline{CP}_{\mu}^{*}(T,{\mathcal{F}},\epsilon)
=\lim_{\delta\rightarrow 0}\ \inf \{
\underline{CP}_{Z}(T,{\mathcal{F}},\epsilon):\mu (Z)\geq 1-\delta\}, \\
\overline{CP}_{\mu}^{*}(T,{\mathcal{F}},\epsilon)
=\lim_{\delta\rightarrow 0}\ \inf \{
\overline{CP}_{Z}(T,{\mathcal{F}},\epsilon):\mu (Z)\geq 1-\delta\}.
\end{gather*}
The \emph{subadditive lower} and \emph{upper capacity measure-theoretic
pressure of $T$} with respect to measure $\mu$ are defined by
\begin{gather}
\label{pressure2}
\underline{CP}_{\mu}^{*}(T,{\mathcal{F}})
=\liminf_{\epsilon\rightarrow 0}
\underline{CP}_{\mu}^{*}(T,{\mathcal{F}},\epsilon),\\
\label{pressure3}
\overline{CP}_{\mu}^{*}(T,{\mathcal{F}})
=\liminf_{\epsilon\rightarrow 0}
\overline{CP}_{\mu}^{*}(T,{\mathcal{F}},\epsilon).
\end{gather}
\end{definition}

%Our main results are the following two theorems.

\begin{TheoremA}\label{ThmA} \rm
Let $(X,T)$ be a TDS and ${\mathcal{F}}=\{ f_n\}_{n\geq 1}$
 a subadditive potential on $X$.  For any $\mu\in
\mathcal{E}_T$ with ${\mathcal{F}}_*(\mu)\neq -\infty$,  we have
\[
P_{\mu}^{*}(T,{\mathcal{F}})
=\underline{CP}_{\mu}^{*}(T,{\mathcal{F}})
=\overline{CP}_{\mu}^{*}(T,{\mathcal{F}})
=P_{\mu}(T,{\mathcal{F}})
=h_{\mu}(T)+{\mathcal{F}}_{*}(\mu).
\]
\end{TheoremA}

\begin{remark}
The results still apply for ${\mathcal{F}}_*(\mu)= -\infty$
if $h_{\mu}(T)<\infty$.
\end{remark}

\begin{remark}
If $\mathcal{F}=\{f_n\}$ is an additive sequence generated by
a continuous function  $\varphi :X\rightarrow \mathbb{R}$,
then we have $P_{\mu}(T,\varphi)=h_{\mu}(T)+\int \varphi {\mathrm{d}}\mu$,
So Theorem A extends the results in \cite{pes1} and \cite{h}
to the subadditive case.
Also, the last equality was proved in \cite{czc}.
\end{remark}

\begin{remark}
By the definition of $P_{\mu}^{*}(T,{\mathcal{F}},\epsilon)$ and
$P_{\mu}^{*}(T,{\mathcal{F}})$, the theorem  gives
$h_{\mu}(T)+{\mathcal{F}}_*(\mu)=\inf \{ P_Z(T,{\mathcal{F}}): \mu
(Z)=1 \}$, as in \cite{pes1}.
We call it the \emph{inverse variational principle}.
\end{remark}

The next theorem says that the infimum in the inverse variational
principle can be attained on certain sets.

\begin{TheoremB}\label{ThmB}
Let $(X,T)$ be TDS and ${\mathcal{F}}=\{f_n\}$  a subadditive
potential on $X$. For any $\mu\in \mathcal{E}_T$ with
${\mathcal{F}}_*(\mu)\neq -\infty$, let
\[
K=\left\{x\in X:\lim_{\epsilon\rightarrow 0}{\limsup_{n\rightarrow\infty}}
\frac{-\log\mu(B_n(x,\epsilon))}{n}=h_\mu(T)
 \  \ \text{and} \   \lim_{n\rightarrow\infty}\frac{1}{n}
f_n(x)={\mathcal{F}}_*(\mu) \right\}.
\]
Then we have
\[
P_{\mu}(T,{\mathcal{F}})
=P_K(T,{\mathcal{F}})
=\underline{CP}_{K}(T,{\mathcal{F}})
=\overline{CP}_{K}(T,{\mathcal{F}}).
\]
\end{TheoremB}

Similar to subadditive sequences, we can study supadditive sequences.
A sequence  $\Phi=\{ \phi_n\}_{n\geq 1}$ of continuous
functions on $X$ is a \emph{supadditive potentials} on $X$, if
\begin{equation}\label{fdefsupadd}
\phi_{n+m}(x)\geq \phi_n(x) +\phi_m(T^n x)
\ \ \text{for all} \ \ x\in X,~ n,m\in {\mathbb{N}}.
\end{equation}

Note that if $\Phi=\{ \phi_n\}_{n\geq 1}$ is a supadditive
sequences, then $-\Phi=\{- \phi_n\}_{n\geq 1}$ is a susadditive
sequences. So the
limit$\lim\limits_{n\rightarrow\infty}(1/n)\phi_n$ exists
$\mu$-almost everywhere for any $\mu\in \M_T$. For $\mu\in
\mathcal{M}_T$, let
 \[ \Phi_{*}(\mu)=\lim\limits_{n\rightarrow\infty}
 \frac{1}{n} \int \phi_n {\mathrm{d}} \mu
=\sup_{n\geq 1} \left\{\frac{1}{n}\int \phi_n{\mathrm{d}} \mu\right\}.
 \]
It is always bounded below by $\int \phi_1{\mathrm{d}} \mu$.

With the sequences, we can define supadditive measure-theoretic pressure
$P_{\mu}(T,\Phi)$ and other quantities
$P_{\mu}^{*}(T,\Phi),\underline{CP}_{\mu}^{*}(T,\Phi)$
 and $\overline{CP}_{\mu}^{*}(T,\Phi)$, etc.
(see Section~\ref{Ssupadd} for precise definitions.)

In \cite{ban}, the authors gave the variational principle for
supadditive topological pressure for $C^1$ average conformal
expanding maps $T$ where the potentials are of the form
$\{\phi_n(x)\}=\{-t\log \|DT^n(x)\|\}$ with $t>0$. However, it is
still open whether variational principle holds for supadditive
topological pressure for a general TDS. Here we show that the
supadditive measure-theoretic pressure has similar formalisms as
subadditive measure-theoretic pressure.

\begin{TheoremC}\label{dl42}
Let $(X,T)$ be a TDS and $\Phi=\{\phi_n\}$ a supadditive
potential on $X$. For any $\mu\in \mathcal{E}_T$,  we have
\[
P_{\mu}^{*}(T,\Phi)
=\underline{CP}_{\mu}^{*}(T,\Phi)
=\overline{CP}_{\mu}^{*}(T,\Phi)
=P_{\mu}(T,\Phi)
=h_{\mu}(T)+\Phi_{*}(\mu).
\]
\end{TheoremC}

\begin{remark}\label{zhusup-additive}
For each $\mu\in \mathcal{E}_T$, from the definition of
$P_{\mu}^{*}(T,\Phi,\epsilon)$ and $P_{\mu}^{*}(T,\Phi)$
given in Section~\ref{Ssupadd}, the theorem  gives
the inverse variational principle
$h_{\mu}(T)+\Phi_*(\mu)=\inf \{ P_Z(T,\Phi): \mu (Z)=1 \}$.
\end{remark}

\begin{TheoremD}\label{dl42}
Let $(X,T)$ be TDS, and $\Phi=\{\phi_n\}$ a supadditive potential
on $X$. For any $\mu\in \mathcal{E}_T$, let
\[
K=\{x\in X:\lim_{\epsilon\rightarrow0}
{\limsup_{n\rightarrow\infty}}\frac{-\log\mu(B_n(x,\epsilon))}{n}=h_\mu(T)
\ \ \text{and}\ \ \lim_{n\rightarrow\infty}\frac{1}{n} \phi_n(x)
=\Phi_*(\mu)\}.
\]
Then we have
\[
P_{\mu}(T,\Phi)=P_K(T,\Phi)=\underline{CP}_{K}(T,\Phi)=\overline{CP}_{K}(T,\Phi).
\]
\end{TheoremD}

%%%%%%%%%%%%%%%%%%%%%%%%%%%%%%%%%%%%%%%%%%%%%%%%
%%%%%%%%%%%%%%%%%%%%%%%%%%%%%%%%%%%%%%%%%%%%%%%%
\section{Subadditive measure-theoretic pressures}\label{Ssubadd}
\setcounter{equation}{0}
%%%%%%%%%%%%%%%%%%%%%%%%%%%%%%%%%%%%%%%%%%%%%%%%
%%%%%%%%%%%%%%%%%%%%%%%%%%%%%%%%%%%%%%%%%%%%%%%%

We start with the section by some properties of pressures for
subadditive potentials.

\begin{proposition}\label{Psubadd}
Let $(X,T)$ be a TDS and $\mathcal{F}=\{f_n\}$ a subadditive potential.
Then the following properties hold:
\begin{enumerate}
\item[(i)] ${\mathcal P}_{Z_1}(T,{\mathcal{F}})
\leq {\mathcal P}_{Z_2}(T,{\mathcal{F}})$
if $Z_1\subset Z_2$,
where ${\mathcal P}$ is $P$, $\underline{CP}$ or $\overline{CP}$;

\item[(ii)] $P_{Z}(T,{\mathcal{F}})=\sup_{i\geq 1}
P_{Z_i}(T,{\mathcal{F}})$ and ${\mathcal P}_{Z}(T,{\mathcal{F}})
\geq \sup_{i\geq 1}{\mathcal P}_{Z_i}(T,{\mathcal{F}})$, where
$Z=\bigcup_{i\geq 1}Z_i$, and ${\mathcal P}$ is  $\underline{CP}$
or $\overline{CP}$;

\item[(iii)] $P_{Z}(T,{\mathcal{F}})
\leq \underline{CP}_{Z}(T,{\mathcal{F}})
\leq  \overline{CP}_{Z}(T,{\mathcal{F}})$
for any subset $Z\subset X$;

\item[(iv)] $P_{\mu}^{*}(T,{\mathcal{F}},\epsilon)\leq
\underline{CP}_{\mu}^{*}(T,{\mathcal{F}},\epsilon)\leq
\overline{CP}_{\mu}^{*}(T,{\mathcal{F}},\epsilon)$,
and $P_{\mu}^{*}(T,{\mathcal{F}})\leq
\underline{CP}_{\mu}^{*}(T,{\mathcal{F}})\leq
\overline{CP}_{\mu}^{*}(T,{\mathcal{F}})$.
\end{enumerate}
\end{proposition}

\begin{proof}
(i) and (ii) are dirctly follow from the definition.
And (iii) is immediately from similarly arguments as in
\cite[Theorem 1.4 (a)]{ba}.
(iv) follows from (iii) immediately by the definition.
\end{proof}

Recall the Brin-Katok's theorem for local entropy (see \cite{brin}),
which says that if $\mu\in \mathcal{M}_T$, then
for $\mu$-almost every $x\in X$,
\begin{equation}\label{fent}
h_\mu(x, T)
=\lim_{\epsilon\to 0}\liminf_{n\rightarrow\infty} -\frac 1 n
\log \mu(B_n(x,\epsilon))
=\lim_{\epsilon\to 0}\limsup_{n\rightarrow\infty} -\frac 1 n
\log \mu(B_n(x,\epsilon)).
\end{equation}
Moreover, if $\mu\in \mathcal{E}_T$, then for $\mu-$almost every $x\in X$,
$h_\mu(x, T)=h_\mu(T)$, and for each $\epsilon >0$,
the following two limits are constants almost everywhere:
\[
\liminf_{n\rightarrow\infty} -\frac 1 n \log
\mu(B_n(x,\epsilon)),~~\limsup_{n\rightarrow\infty} -\frac 1 n
\log \mu(B_n(x,\epsilon)).
\]

\begin{proof}[Proof of Theorem A]
The last equality is a direct consequence of Theorem 2.3 in \cite{czc}.
By Proposition~\ref{Psubadd}, we only need to prove
$\overline{CP}_{\mu}^{*}(T,{\mathcal{F}})
\leq h_\mu (T)+ {\mathcal{F}}_{*}(\mu)$
and
$P_{\mu}^{*}(T,{\mathcal{F}})\geq h_\mu (T)+\mathcal{F}_{*}(\mu)$.

For $\mu\in \mathcal{E}_T$,  we first assume $h_\mu (T)$ is finite
and set $h=h_\mu (T)\geq 0$.

Take $\delta>0$.
Fix a positive integer $k$ and a small number $\eta >0$.

Take $\epsilon_\eta>0$ such that if $\epsilon\in (0,\epsilon_\eta]$,
then for $\mu$-almost every $x\in X$,
\[
h-\eta/2 \le \liminf_{n\rightarrow\infty} -\frac 1 n \log
\mu(B_n(x,\epsilon)) \le \limsup_{n\rightarrow\infty} -\frac 1 n
\log \mu(B_n(x,\epsilon))\le h+\eta/2.
\]
This is possible because of (\ref{fent}).
Take $0<\epsilon\le \min\{\epsilon_\eta,\epsilon_0\}$, where $\epsilon_0$
is given in Lemma~\ref{subadditive}.
Hence, for $\mu$-almost every $x\in X$, there exists a number $N_1(x)>0$
such that for any $n\geq N_1(x)$,
\begin{equation}\label{ds410}
\left|\frac{1}{n} \log\mu(B_n(x,\epsilon /2))+h\right|\leq \eta.
\end{equation}

By the Birkhoff ergodic theorem,  %(see e.g. \cite{w2})
for $\mu$-almost every $x\in X$,
there exists a number $N_2(x)>0$ such that for any $n\geq N_2(x)$,
\begin{equation}\label{ds49}
\left|\frac{1}{n} \sum_{i=0}^{n-1}\frac{1}{k}f_k(T^ix)
 -\int\frac{1}{k}f_k {\mathrm{d}}\mu\right|
\leq \eta.
\end{equation}

Given $N>0$, set $K_N=\{ x\in X: N_1(x),N_2(x)\leq N \}$.
We have that $K_N\subset K_{N+1}$, and $\cup_{N\geq 0} K_N$ is a set of full
measure.  Therefore, one can find $N_0>0$ for which $\mu (K_{N_0})>1-\delta$.
Fix a number $N>N_0$.
Then for any $n>N_0$ and any point $x\in K_N$, by (\ref{ds49}) and
Lemma~\ref{subadditive} below we have
\[
\sup_{y\in B_{n}(x,\epsilon) } f_n( y)\leq n\int \frac{1}{k} f_k
{\mathrm{d}}\mu +2n\eta+C ,
\]
where $C$ is a constant given in Lemma~\ref{subadditive}.

Let $E$ be a maximal $(n,\epsilon)$-separated subset of $K_N$,
then $K_N \subseteq \cup_{x\in E} B_n(x,\epsilon)$.
Furthermore, the balls $\{B_n(x, \epsilon/2): x\in E\}$ are
pairwise disjoint
and by (\ref{ds410}) the cardinality of $E$ is less than or equal to
$\exp n(h+\eta)$.
Therefore, we have
\begin{eqnarray*}
\Lambda(K_N,{\mathcal{F}},n,\epsilon)&\leq &\sum_{x\in E} \exp
(\sup_{y\in
B_{n}(x,\epsilon) } f_n( y))\\
&\leq& \exp n(h+\eta) \cdot
\exp\left[ n(\int \frac{1}{k}f_k {\mathrm{d}}\mu+2\eta)+C\right]\\
&=&\exp\left[ n(\int \frac{1}{k}f_k {\mathrm{d}}\mu+h +3\eta )+C\right].
\end{eqnarray*}
From (\ref{sup2}), we have
\[
\overline{CP}_{K_N}(T,{\mathcal{F}},\epsilon)\leq \int
\frac{1}{k}f_k {\mathrm{d}}\mu+h +3\eta.
\]
Since $\mu (K_N)\geq 1-\delta$, we have
\[
\overline{CP}_{\mu}^{*}(T,{\mathcal{F}},\epsilon) \leq \int
\frac{1}{k}f_k {\mathrm{d}}\mu+h +3\eta.
\]
Let $\epsilon\rightarrow 0$ and $k\rightarrow\infty$ in the inequality,
and by the arbitrariness of $\eta$, we get
$\overline{CP}_{\mu}^{*}(T,{\mathcal{F}})\leq h+ {\mathcal{F}}_{*}(\mu)$.

To prove the other inequality, it is sufficient to prove that
$P_{Z}(T,{\mathcal{F}})\geq h+{\mathcal{F}}_{*}(\mu)$
for any subset $Z\subseteq X$ of full $\mu$-measure.

Take $\eta>0$ and $\delta \in (0,1/2)$, and denote
$\lambda =h+\mathcal{F}_{*}(\mu)-2\eta$.

Let
\[
K=\left\{ x\in X :
\lim_{\epsilon\rightarrow0}{\limsup_{n\rightarrow\infty}}
\frac{-\log\mu(B_n(x,\epsilon))}{n}=h_\mu(T)
\ \ \text{and} \ \ \lim_{n\rightarrow\infty}\frac{1}{n}
f_n(x)={\mathcal{F}}_{*}(\mu) \right\}.
\]
Put $K'=K\cap Z$.  By the Brin-Katok's theorem for local entropy
and the subadditive ergodic theorem, we have $\mu(K)=1$ and then
$\mu (K')=1$. For $\epsilon \in (0,\epsilon_\eta]$, there exists a
set $K_1\subset K'$ with $\mu (K_1)> 1-\frac{\delta}{2}$ and
$N_1
>0$ such that for any $x\in K_1$ and $ n\geq N_1$, we have
\[
\mu (B_n(x,2\epsilon)) \le \exp (-n(h-\eta)).
\]
By the subadditive ergodic theorem, there exists a set $K_2\subset K'$
with $\mu (K_2)> 1-\frac{\delta}{2}$ and $N_2 >0$ such that for
any $x\in K_2$ and $ n\geq N_2$, we have
\[
\left|\frac{1}{n} f_n(x)-{\mathcal{F}}_{*}(\mu)\right|< \eta.
\]

Put $\widetilde{K}=K_1\cap K_2\subset K'$, $N\ge \max \{ N_1, N_2 \}$.
Clearly $\mu (\widetilde{K})> 1-\delta$.
%and we can assume $N$ large enough when it is necessary.
We may assume further that $\widetilde{K}$ is compact since otherwise
we can approximate it from within by a compact subset.
Take an open cover $\Gamma=\{B_{n_i}(x_i,\epsilon)\}_i$ of
$\widetilde{K}$ with $n(\Gamma)\geq N$. Since $\widetilde{K}$ is
compact, we may assume that the cover is finite and consists
of $B_{n_1}(x_1,\epsilon),\cdots,B_{n_l}(x_l,\epsilon)$.

For each $i=1,\cdots,l$, we choose
$y_i\in \widetilde{K}\cap B_{n_i}(x_i,\epsilon)$.
Hence, $B_{n_i}(x_i,\epsilon)\subset B_{n_i}(y_i,2\epsilon)$,
and $\{B_{n_i}(y_i,2\epsilon)\}_i$ form a cover of
$\widetilde{K}$ as well.
Now we have
\begin{equation*}
\begin{split}
&\sum_{B_{n_i}(x_i,\epsilon)\in \Gamma}
\exp\bigl(-n_i\lambda +\sup_{y\in B_{n_i}(x_i,\epsilon)}f_{n_i}(y)\bigr)
\geq \sum_{i=1}^{l} \exp \bigl(-n_i\lambda +f_{n_i}(y_i)\bigr)\\
\ge &\sum_{i=1}^{l}\exp\bigl(-n_i\lambda+n_i(\mathcal{F}_{*}(\mu)-\eta)\bigr)
=\sum_{i=1}^{l}\exp\bigl(-n_i(h-\eta)\bigr)
\geq \sum_{i=1}^{l} \mu (B_{n_i}(y_i,2\epsilon))
\geq 1-\delta \geq \frac{1}{2}.
\end{split}
\end{equation*}
Note that the inequality  holds for any cover
$\Gamma=\{B_{n_i}(x_i,\epsilon)\}_i$ of $\widetilde{K}$.  Hence
\[
M(\widetilde{K},{\mathcal{F}},\lambda,N,\epsilon)\geq \frac{1}{2}.
\]
Thus $m(\widetilde{K},{\mathcal{F}},\lambda, \epsilon)\geq 1/2$.
It means that
\begin{eqnarray*}
P_{\widetilde{K}}(T,{\mathcal{F}},\epsilon) \geq \lambda
=h+\mathcal{F}_{*}(\mu)-2\eta,
\ \ \text{and} \ \
P_{\widetilde{K}}(T,{\mathcal{F}})\geq  h+\mathcal{F}_{*}(\mu)-2\eta.
\end{eqnarray*}
Using Proposition~\ref{Psubadd} and arbitrariness of $\eta$, we have
\begin{equation}\label{fge}
P_{Z}(T,{\mathcal{F}})\geq P_{\widetilde{K}}(T,{\mathcal{F}})\geq
 h+\mathcal{F}_{*}(\mu).
\end{equation}
So by definition we get
$P_{\mu}^{*}(T,{\mathcal{F}})\geq h+\mathcal{F}_{*}(\mu)$.

When $h_\mu (T)=+\infty$, modify subtly the proof of the second
inequality, we can easily have
$P_{\mu}^{*}(T,{\mathcal{F}})=+\infty$. Thus we finish the proof
of the theorem.
\end{proof}

\begin{proof}[Proof of Theorem B]
As in the proof of Theorem~A, set $h=h_\mu(f)\geq 0$.
Fix an integer $k>0$ and a small real number $\eta>0$.

For $\mu$-almost every $x\in X$, there exists a number
$N_1(x)>0$ such that for any $n\geq N_1(x)$,
\begin{eqnarray*}
\left|\frac{1}{n} \sum_{i=0}^{n-1}\frac{1}{k}f_k(T^ix)-\int \frac{1}{k}
f_k {\mathrm{d}}\mu\right|\leq \eta.
\end{eqnarray*}
For some $\epsilon \in (0, \epsilon_\eta]$, where $\epsilon_\eta$
is chosen in the same way as that in the proof of Theorem A,
for $\mu$-almost every $x\in X$ there exists a number $N_2(x)>0$
such that for any $n\geq N_2(x)$,
\begin{eqnarray*}
\left|\frac{1}{n} \log\mu(B_n(x,\epsilon/2))+h\right|\leq \eta.
\end{eqnarray*}
Given $N>0$, set $K_N=\{ x\in K: N_1(x), N_2(x)\leq N\}$,
we have $K_N\subset K_{N+1}$, and $\cup_{N\geq 0} K_N=K$.
Similarly, given $\delta >0$, we can find $N_0>0$ for which $\mu
(K_{N_0})>1-\delta$.

Fix a number $N\geq N_0$, as in the proof of Theorem~A we get that
$\overline{CP}_{K_N}(T,{\mathcal{F}},\epsilon)\leq \int
\frac{1}{k}f_k {\mathrm{d}}\mu+h +3\eta$.
Letting $k\rightarrow\infty$, $\epsilon\rightarrow 0$,
since $\eta$ is arbitrary, we get
\[
\overline{CP}_{K_N}(T,{\mathcal{F}})\leq h+\mathcal{F}_{*}(\mu).
\]
Letting $N\rightarrow \infty$, we have that
$\overline{CP}_{K}(T,{\mathcal{F}})\leq h+\mathcal{F}_{*}(\mu)$.

The inequality $P_K(T,{\mathcal{F}})\geq h+\mathcal{F}_{*}(\mu)$
is contained in (\ref{fge}) since $Z$ is an arbitrary set with full
measure, and here we have $\mu (K)=1$.

By Theorem~A and Proposition~\ref{Psubadd}, we get the desired result.
\end{proof}

\begin{lemma}\label{subadditive}
Let $(X,T)$ be a  TDS and $\mathcal{F}=\{f_n\}$ a subadditive
potential.  Fix any positive integer $k$, then for any $\eta>0$,
there exist an $\epsilon_0>0$ such that for any $0<\epsilon< \epsilon_0$
we have
 \[
 \sup_{y\in B_n(x,\epsilon)} f_n(y) \leq \sum_{i=0}^{n-1}
 \frac{1}{k}f_k(T^ix)+n\eta +C,
 \]
 where $C=C_k$ is a constant independent of $\eta$ and $\epsilon$.
 \end{lemma}

\begin{proof}
Fix a positive integer $k$, $\frac{1}{k}f_k(x)$ is a  continuous function.
Hence, for any $\eta >0$, there exist an $\epsilon_0 >0$
such that for any $0<\epsilon  <\epsilon_0$,
\begin{eqnarray} \label{ds21}
d(x,y)<\epsilon
\Rightarrow d\left(\frac{1}{k}f_k(x), \frac{1}{k}f_k(y)\right)<\eta.
\end{eqnarray}

For each $n$, we rewrite $n$ as $n=sk+l$, where $s\ge 0$, $0\leq l<k$.
Then for any integer $0\leq j<k$, we have
\[
f_n(x)\leq f_j(x)+f_k(T^jx)+\cdots +
f_k(T^{(s-2)k}T^jx)+f_{k+l-j}(T^{(s-1)k}T^jx),
\]
where we take $f_0(x)\equiv 0$.
Let $C_1=\max_{j=1,\cdots 2k}\max_{x\in X}|f_j(x)|$.
Summing over $j$ from $0$ to $k-1$, we have
\[
kf_n(x)\leq 2kC_1 + \sum_{i=0}^{(s-1)k-1} f_k(T^ix).
\]
Hence
\begin{eqnarray} \label{bds22}
f_n(x)\leq 2C_1+ \sum_{i=0}^{(s-1)k-1} \frac{1}{k}f_k(T^ix)\leq
4C_1+ \sum_{i=0}^{n-1} \frac{1}{k}f_k(T^ix).
\end{eqnarray}
Set $C=4C_1$.  By (\ref{ds21}) we have that
\[
\sup_{y\in B_n(x,\e)} f_n(y) \leq \sup_{y\in B_n(x,\e)}
(C+\sum_{i=0}^{n-1} \frac{1}{k}f_k(T^iy))\leq \sum_{i=0}^{n-1}
 \frac{1}{k}f_k(T^ix)+n\eta +C.
\]
This completes  the proof of the lemma.
\end{proof}

%%%%%%%%%%%%%%%%%%%%%%%%%%%%%%%%%%%%%%%%%%%%%%%%
%%%%%%%%%%%%%%%%%%%%%%%%%%%%%%%%%%%%%%%%%%%%%%%%
\section{Supadditive measure-theoretic pressures}\label{Ssupadd}
\setcounter{equation}{0}
%%%%%%%%%%%%%%%%%%%%%%%%%%%%%%%%%%%%%%%%%%%%%%%%
%%%%%%%%%%%%%%%%%%%%%%%%%%%%%%%%%%%%%%%%%%%%%%%%

Recall that supadditive sequence is defined in \eqref{fdefsupadd}

\begin{definition}\label{defPmusup}
Let $\Phi=\{\phi_n\}$ be a given supadditive potential.
For $\mu\in \mathcal{E}_T$, $0< \delta < 1$, $n\geq 1$, and $\epsilon >0$,
put
\begin{eqnarray*}
&&P_{\mu}(T,\Phi,n,\epsilon,\delta)
=\inf \bigl\{\sum\limits_{x\in
 F} e^{\phi_{n}(x)}
\mid F\ \text{is an}\ (n,\epsilon,\delta)\text{-spanning set} \bigr\},\\
&&P_{\mu}(T,\Phi,\epsilon,\delta)=\limsup\limits_{n\rightarrow\infty}
 \frac{1}{n} \log  P_{\mu}(T,\Phi,n,\epsilon,\delta),\\
&&P_{\mu}(T,\Phi,\delta)=\lim\limits_{\epsilon\rightarrow
 0} P_{\mu}(T,\Phi,\epsilon,\delta),\\
 &&P_{\mu}(T,\Phi)=\lim\limits_{\delta\rightarrow
 0} P_{\mu}(T,\Phi,\delta),
 \end{eqnarray*}
where $P_{\mu}(T,\Phi)$ is called the
\emph{supadditive measure-theoretic pressure} of $T$ with respect to $\Phi$.
\end{definition}

Note that there is small difference between the definitions of
$P_{\mu}(T,{\mathcal{F}},n,\epsilon,\delta)$ and
$P_{\mu}(T,\Phi,n,\epsilon,\delta)$.
This difference makes it
possible to remove the tempered variation assumptions on the
potentials as in \cite{z1,z2}.

\begin{remark}\label{RmkPmusup}
Similar to Remark~\ref{RmkPmu},
$P_{\mu}(T,\Phi,\delta)$ also increases with $\delta$ and
therefore the limit in the last formula exists.
Moreover, by Proposition~\ref{Pmusup} below
$P_{\mu}(T,\Phi,\delta)$ is independent of $\delta$.
\end{remark}

Recall that if $\Gamma=\{B_{n_i}(x_i,\epsilon)\}_i$ is a cover of
a subset $Z\subseteq X$, where $\epsilon>0$, then $n(\Gamma)=\min_i\{n_i\}$.

\begin{definition}\label{defPmustarsup}
Let $\Phi=\{\phi_n\}$ be a given supadditive potential.
For $s\geq 0$, define
\begin{eqnarray*}
M(Z,\Phi,s,N,\epsilon)=\inf_\Gamma\sum_i\exp(-sn_i+ \phi_{n_i}(
x_i)),
\end{eqnarray*}
where the infimum is taken over all covers
$\Gamma=\{B_{n_i}(x_i,\epsilon)\}_i$ of $Z$ with $n(\Gamma)\geq N$.
Then let
\begin{gather}%\label{other2}
m(Z,\Phi,s,\epsilon)=\lim_{N\rightarrow\infty}M(Z,\Phi,s,N,\epsilon),\\
P_{Z}(T,\Phi,\epsilon)
=\inf \{ s: m(Z,\Phi,s,\epsilon)=0 \}
=\sup \{ s: m(Z,\Phi,s,\epsilon)=+\infty \}, \\
P_{Z}(T,\Phi)=\liminf_{\epsilon\rightarrow 0}P_{Z}(T,\Phi,\epsilon),
\end{gather}
where $P_{Z}(T,{\Phi})$ is called a \emph{supadditive
topological pressure} of $T$ on the set $Z$ (w.r.t. $\Phi$).

Further, for $\mu \in \mathcal{M}_T$, put
\begin{gather}
P_{\mu}^{*}(T,{\Phi},\epsilon)
=\inf \{ P_{Z}(T,{\Phi},\epsilon):\mu (Z)=1\}, \notag \\
P_{\mu}^{*}(T,{\Phi})=\liminf_{\epsilon\rightarrow
0}P_{\mu}^{*}(T,{\Phi},\epsilon), \label{pressure1}
\end{gather}
where $P_{\mu}^{*}(T,{\Phi})$ is called a \emph{supadditive
measure-theoretic pressure} of $T$ with respect to $\mu$.
\end{definition}

\begin{definition}\label{defCPsup}
Put
\[
\Lambda(Z,\Phi,N,\epsilon)
=\inf_\Gamma\sum_{B_{N}(x,\epsilon)\in\Gamma}\exp(\phi_N(x)),
\]
where the infimum is taken over all covers  $\Gamma$ of $Z$ with
$n_i= N$ for all $i$.
And then we set
\begin{gather}\label{othersup1}
 \underline{CP}_{Z}(T,\Phi,\epsilon)
=\liminf_{N\rightarrow\infty} \frac{1}{N}\log\Lambda(Z,\Phi,N,\epsilon),\\
\label{othersup2}
\overline{CP}_{Z}(T,\Phi,\epsilon)
=\limsup_{N\rightarrow\infty}\frac{1}{N}\log\Lambda(Z,\Phi,N,\epsilon).
\end{gather}
For $\mu \in \mathcal{M}_T$, define
\begin{gather*}
\underline{CP}_{\mu}^{*}(T,{\Phi},\epsilon)
=\lim_{\delta\rightarrow 0}\ \inf \{
\underline{CP}_{Z}(T,{\Phi},\epsilon):\mu (Z)\geq 1-\delta\}, \\
\overline{CP}_{\mu}^{*}(T,{\Phi},\epsilon)
=\lim_{\delta\rightarrow 0}\ \inf \{
\overline{CP}_{Z}(T,{\Phi},\epsilon):\mu (Z)\geq 1-\delta\}.
\end{gather*}
The \emph{supadditive lower} and \emph{upper capacity measure-theoretic
pressure of $T$} with respect to measure $\mu$ are defined by
\begin{gather}
\label{pressuresup2}
\underline{CP}_{\mu}^{*}(T,{\Phi})
=\liminf_{\epsilon\rightarrow 0}
\underline{CP}_{\mu}^{*}(T,{\Phi},\epsilon),\\
\label{pressuresup3}
\overline{CP}_{\mu}^{*}(T,{\Phi})
=\liminf_{\epsilon\rightarrow 0}
\overline{CP}_{\mu}^{*}(T,{\Phi},\epsilon).
\end{gather}
\end{definition}

It is easy to see that all the pressures of supadditive potentials
defined as above have the same properties
of the corresponding pressures of subadditive potentials.
We state it here, whose proof is similar and left to the reader.

\begin{proposition}\label{Psupadd}
All the properties stated in Proposition~\ref{Psubadd} are true
if we replace the subadditive potential $\mathcal{F}$
by a supadditive potential $\Phi$.
\end{proposition}

\begin{proof}[Proof of Theorem C]
The last equality is proved in Proposition~\ref{Pmusup} below.
So by Proposition~\ref{Psupadd}, we only need to prove
$\overline{CP}_{\mu}^{*}(T,{\Phi})
\leq h_\mu (T)+ {\mathcal{F}}_{*}(\mu)$
and
$P_{\mu}^{*}(T,{\Phi})\geq h_\mu (T)+\mathcal{F}_{*}(\mu)$.

For $\mu\in \mathcal{E}_T$, we first assume $h_\mu (T)$ is finite
and set $h=h_\mu (T)\geq 0$.

Fix a  small number $\eta >0$ and an $\epsilon \in (0,\epsilon_\eta]$,
where $\epsilon_\eta$ is determined in a similar way as in the proof
of Theorem~A.  Hence, by the Brin-Katok theorem (see \cite{brin})
for local entropy,  for $\mu$-almost every $x\in X$ there exists
a number $N_1(x)>0$ such that for any $n\geq N_1(x)$, we have
\begin{equation} \label{fThmC1}
\left|\frac{1}{n} \log\mu(B_n(x,\epsilon /2))+h\right|\leq \eta.
\end{equation}
Since $-\Phi=\{-\phi_n\}$ is a subadditive sequence,
by the subadditive ergodic theorem,
for $\mu$-almost every $x\in X$ there exists a number $N_2(x)>0$
such that for any $n\geq N_2(x)$, we have
\begin{equation} \label{fThmC2}
\left|\frac{1}{n} \phi_n(x)-\Phi_{*}(\mu)\right|\leq \eta
\end{equation}
Given $N>0$, set $K_N=\{x\in X: N_1(x),N_2(x)\leq N \}$. We have
that $K_N\subset K_{N+1}$, and $\cup_{N\geq 0} K_N$ is a set of full
measure. Therefore, given $\delta >0$, we can find $N_0>0$ for
which $\mu (K_{N_0})>1-\delta$.

Fix a number $N\geq N_0$. Let $E$ be a maximal $(n,\epsilon)$-separated
subset of $K_N$, then $K_N\subseteq \cup_{x\in E} B_n(x,\epsilon)$.
Furthermore, the balls $\{B_n(x, \epsilon/2): x\in E\}$ are
pairwise disjoint.
By (\ref{fThmC1}) the cardinality of $E$ is less than or equal to
$\exp n(h+\eta)$. Therefore, by (\ref{fThmC2}) we have
\begin{equation*}
\begin{split}
\Lambda(K_N,\Phi,n,\epsilon)
\leq &\sum_{x\in E} \exp (\phi_n( x))
\leq \exp n(h+\eta) \cdot \exp n(\Phi_*(\mu)+\eta) \\
=&\exp n(\Phi_*(\mu)+h +2\eta ).
\end{split}
\end{equation*}
It follows
\[
\overline{CP}_{K_N}(T,\Phi,\epsilon)\leq\Phi_*(\mu)+h +2\eta.
\]
Since $\mu (K_N)\geq 1-\delta$ and $\mu (K_N)\to 1$ as $N\to \infty$,
we have
\[
\overline{CP}_{\mu}^{*}(T,\Phi,\epsilon) \leq \Phi_*(\mu)+h +2\eta.
\]
Let $\epsilon\rightarrow 0$ and take limit.
By arbitrariness of $\eta$, we have
$\overline{CP}_{\mu}^{*}(T,\Phi)\leq h+ \Phi_{*}(\mu)$.

To prove the other inequality, it is sufficient to prove that
$P_{Z}(T,\Phi)\geq h+\Phi_{*}(\mu)$ for any subset $Z\subseteq X$
of full $\mu$-measure.
Fix a subset $Z$ with full measure and a positive integer $k$.
Note that for an additive sequence $\displaystyle {\mathcal{F}}=\{f_n\}
=\bigl\{\sum_{i=1}^{n}\frac {1}{k}\phi_k\circ T^{i-1}\bigr\}$,
$P_Z(T, {\mathcal{F}})$ becomes the pressure $P_Z(T, f_1)$
of $f_1=(1/k)\phi_k$.
So if we apply the same arguments as in the proof of Theorem~A for
the sequence, we get the inequality as in (\ref{fge}),
that is,
\[
P_{Z}(T,\frac 1 k \phi_k)
\geq h+\int \frac 1 k \phi_k \mathrm{d}\mu.
\]
Using lemma \ref{sup-additive2} below, we have
\[
P_{Z}(T,\Phi)\geq h+\int \frac 1 k \phi_k \mathrm{d}\mu.
\]
The arbitrariness of $k$ implies that $P_{Z}(T,\Phi)\geq
h+\Phi_{*}(\mu)$.

If $h_\mu (T)=+\infty$, we can easily have
$P_{\mu}^{*}(T,\Phi)=+\infty$. Thus we finish the proof of the
theorem.
\end{proof}

\begin{proof}[Proof of Theorem D]
As in the proof of Theorem~C, set $h=h_\mu (f)\geq 0$.
By the same arguments, there is a sequence of subsets $\{K_N\}_{N\geq 1}$
such that  $K_N\subset K_{N+1}$ and $\cup_{N\geq 0} K_N=K$.
Moreover,
\[
\overline{CP}_{K_N}(T,\Phi)\leq h+\Phi_{*}(\mu)
\]
for all sufficiently large $N$. Letting $N\rightarrow \infty$, we
get that $\overline{CP}_{K}(T,\Phi)\leq h+\Phi_{*}(\mu)$.

By Theorem~C, the reverse inequality $P_K(T,\Phi)\geq h+\Phi_{*}(\mu)$
is immediate since $\mu (K)=1$.
Hence Theorem~C and Proposition~\ref{Psupadd} implies the desired
results.
\end{proof}

\begin{proposition}\label{Pmusup}
Let $(X,T)$ be a TDS, and $\Phi=\{\phi_n\}$ a supadditive
potential on $X$. For $\mu\in \mathcal{E}_T$, we have
 \[
 P_{\mu}(T,\Phi)
=\lim\limits_{\epsilon\rightarrow 0} \liminf\limits_{n\rightarrow\infty}
\frac{1}{n} \log P_{\mu}(T,\Phi,n,\epsilon,\delta)
=h_{\mu}(T)+\Phi_{*}(\mu).
\]
\end{proposition}

\begin{proof}
Fix a positive integer $k$ and a small number $\delta>0$.
Take $\eta>0$.  Let $\epsilon_0$ be as in Sublemma~\ref{supadditive}.
Then by Lemma~\ref{non-additive}, for any $\epsilon\in (0,\epsilon_0]$,
$n>0$, we can get
\[
 P_{\mu}(T,\Phi,n,\e,\d)
\geq e^{-n\eta-C}P_{\mu}(T,\frac{\phi_k}{k},n,\e,\d).
\]
By Theorem 2.1 in \cite{h},
\[
\lim_{\epsilon\to 0}\liminf_{n\to \infty}
\frac{1}{n} \log P_{\mu}(T,\frac{\phi_k}{k},n,\e,\d)
=h_{\mu}(T)+\int_{X}\frac{\phi_k}{k} \mathrm{d}\mu.
 \]
Hence by definition, we have
\[
 P_{\mu}(T,\Phi)
\ge \lim\limits_{\epsilon\rightarrow 0} \liminf\limits_{n\to\infty}
\frac{1}{n} \log P_{\mu}(T,\Phi,n,\epsilon,\delta) -\eta
\ge h_{\mu}(T)+\int_{X}\frac{\phi_k}{k}\mathrm{d}\mu -\eta.
 \]
Note that $\eta$ is arbitrary. By letting $k\rightarrow\infty$ we get
\begin{eqnarray} \label{bds42}
P_{\mu}(T,\Phi)
\geq \lim\limits_{\epsilon\rightarrow 0} \liminf\limits_{n\to\infty}
\frac{1}{n} \log P_{\mu}(T,\Phi,n,\epsilon,\delta)
\geq h_{\mu}(T)+\Phi_{*}(\mu).
\end{eqnarray}

Now we only need to prove the reversed inequality
$P_{\mu}(T,\Phi)\leq h_{\mu}(T)+\Phi_{*}(\mu)$.

Take $\eta>0$.  For each $N$, take a set $K_N$ as in the proof of
Theorem~C.
Then take $N_0>0$ such that for any $N\ge N_0$, $\mu(K_N)>1-\delta$.
For any $n\ge N$, let $F_n$ be a maximal $(n,\epsilon/2)$-separated
subset of $K_N$.
Then $K_N \subseteq \cup_{x\in F_n} B_n(x,\epsilon)$.
It means that $F_n$ is a $(n,\epsilon,\delta)$-spanning set.

By (\ref{fThmC1}), if $x\in K_N$, then $\mu B_n(x,\epsilon/2)\ge
\exp[-n(h+\eta)]$. Hence, $F_n$ contains at most $\exp[n(h+\eta)]$
elements. By (\ref{fThmC2}), if $x\in K_N$, then $\phi_n(x)\le
n(\Phi_*(\mu)+\eta)$. Now we get
\begin{equation*}
\begin{split}
  \sum\limits_{x\in F_{n}} \exp[\phi_{n}(x)]
\leq & \sum\limits_{x\in F_{n}} \exp [n(\Phi_{*}(\mu)+\eta)]
\leq \exp[n(h+\eta)]\cdot \exp [n(\Phi_{*}(\mu)+\eta)] \\
= &\exp [n(h+\Phi_{*}(\mu)+2\eta)].
\end{split}
\end{equation*}
Therefore,
\[
   P_{\mu}(T,\Phi,n,\epsilon,\delta)
\leq   \exp [n(h+\Phi_{*}(\mu)+2\eta)].
\]
Consequently,
\begin{eqnarray} \label{bds45}
P_{\mu}(T,\Phi,\delta)
\leq h+\Phi_{*}(\mu)+2\eta.
\end{eqnarray}
Since $\delta$ and $\eta$ are arbitrary, we have
\[P_{\mu}(T,\Phi)\leq h_{\mu}(T)+\Phi_{*}(\mu),
\]
which is the desired inequality.
\end{proof}

\begin{lemma}\label{sup-additive2}
Let $(X,T)$ be a TDS and $\Phi=\{\phi_n\}$ a supadditive potential.
Fix any positive integer $k$.  For each subset $Z$ we have
\[
P_{Z}(T,\Phi)\geq  P_{Z}(T,\frac 1 k \phi_k).
\]
\end{lemma}

\begin{proof}
Fix a positive integer $k$.  By Sublemma~\ref{supadditive} below we have
\[
M(Z,\Phi,s,N,\epsilon)\geq e^{-C}M(Z,\frac 1 k \phi_k,s+\eta,N,\epsilon).
\]
Therefore
\[
P_{Z}(T,\frac 1 k \phi_k)\leq P_{Z}(T,\Phi)+\eta.
\]
This immediately implies the desired result.
\end{proof}

\begin{lemma}\label{non-additive}
Let $(X,T)$ be a TDS and $\Phi=\{\phi_n\}$ a supadditive potential.
Let $k$ be any positive integer and $\mu\in \mathcal{E}_T$.
Then for any $\eta >0$,  there exist an $\epsilon_0>0$ such that for any
$0<\epsilon < \epsilon_0$, $n>0$, $\delta\in (0,1]$, we have
 \[
 P_{\mu}(T,\Phi,n,\e,\d)\geq
 e^{-n\eta-C}P_{\mu}(T,\frac{\phi_k}{k},n,\e,\d),
 \]
and therefore\[
P_{\mu}(T,\mathcal{F})\geq
P_{\mu}(T,\frac{f_k}{k}).
\]
\end{lemma}

\begin{remark}
For a subadditive potential $\mathcal{F}=\{f_n\}$,
the result becomes $P_{\mu}(T,\mathcal{F})\leq P_{\mu}(T,\frac{f_k}{k})$.
\end{remark}

\begin{proof}[Proof of Lemma~\ref{non-additive}]
By Sublemma~\ref{supadditive} below, for any small number $\eta >0$,
 there exist an $\epsilon_0>0$ such that for any
$0<\epsilon < \epsilon_0$, $n>0$, $\delta\in (0,1]$, we have
the first inequality of the lemma.
 Therefore
 \[
 P_{\mu}(T,\Phi)\geq P_{\mu}(T,\frac{\phi_k}{k})-\eta.
 \]
 The arbitrariness of $\eta$ immediately yields the desired
 result.
\end{proof}

\begin{sublemma}\label{supadditive}
Let $(X,T)$ be a TDS and $\Phi=\{\phi_n\}$ a supadditive potential.
Fix any positive integer $k$.   Then for any small number
$\eta >0$,  there exist an $\epsilon_0>0$ such that for any
$0<\epsilon < \epsilon_0$ we have
 \[
  \phi_n(x) \geq \sup_{y\in B_n(x,\epsilon)} \sum_{i=0}^{n-1}
 \frac{1}{k}\phi_k(T^iy)-n\eta -C,
 \]
 where $C$ is a constant independent of $\eta$ and $\epsilon$.
\end{sublemma}

\begin{proof}
Fix a positive integer $k$.  Since $\frac{1}{k}\phi_k(x)$ is a
continuous function, for any $\eta >0$, there exist $\epsilon_0 >0$
such that for any $0<\epsilon  <\epsilon_0$,
\begin{eqnarray*}
 d(x,y)<\epsilon
\Rightarrow d(\frac{1}{k}\phi_k(x), \frac{1}{k}\phi_k(y))<\eta.
 \end{eqnarray*}
Using the supadditivity of $\Phi$, as (\ref{bds22}) we have
\[
\phi_n(x)\geq  \sum_{i=0}^{n-1} \frac{1}{k}\phi_k(T^ix)-C,
\]
where $C$ is a constant. Thus
\[
 \phi_n(x) \geq \sup_{y\in B_n(x,\epsilon)} \sum_{i=0}^{n-1}
 \frac{1}{k}\phi_k(T^iy)-n\eta -C,
\]
which is the desired result.
\end{proof}

%%%%%%%%%%%%%%%%%%%%%%%%%%%%%%%%%%%%%%%%%%%%%%%%
%%%%%%%%%%%%%%%%%%%%%%%%%%%%%%%%%%%%%%%%%%%%%%%%
\section{Dimensions of ergodic measure on an average conformal repeller}
\setcounter{equation}{0}
%%%%%%%%%%%%%%%%%%%%%%%%%%%%%%%%%%%%%%%%%%%%%%%%
%%%%%%%%%%%%%%%%%%%%%%%%%%%%%%%%%%%%%%%%%%%%%%%%

In this section we give an application of the nonadditive
measure-theoretic pressures to average conformal repellers defined
in \cite{ban}.
We use the relations proved in the previous sections
among the pressures, entropy, and limits of the nonadditive potentials
to prove that the Hausdorff and box dimension of an average conformal
repeller is equal to the Hausdorff dimension
of an ergodic measure support on it.

Let $M$ be an $m$-dimensional smooth Riemannian manifold. Let $U$
be an open subset of $M$ and $f:U\rightarrow M$ be a $C^1$ map.
Suppose $J\subset U$ is a compact $f$-invariant subset. And let
$\mathcal{M}(f|_J)$ and $\mathcal{E}(f|_J)$ denote the set of all
$f-$invariant measures and the set of all ergodic measures
supported on $J$ respectively.

For $x\in M$ and $v\in T_xM$, the Lyapunov exponent of $v$ at $x$
is the limit
\begin{eqnarray*}
\chi(x,v)=\lim_{n\to\infty}\frac 1n\log \|Df_x^n(v)\|
\end{eqnarray*}
if the limit exists. By the Oseledec multiplicative ergodic
theorem \cite{osel}, for $\mu$-almost every point $x$, every
vector $v\in T_xM$ has a Lyapunov exponent, and they can be
denoted by $\lambda_1(x)\leq \lambda_2(x)\leq \cdots \leq
\lambda_m(x)$. Since the Lyapunov exponents are $f$-invariant, if
$\mu$ is ergodic, then we can define Lyapunov exponents
$\lambda_1(\mu)\leq \lambda_2(\mu)\leq \cdots \leq
 \lambda_m(\mu)$, where $m=\mathrm{dim}M$, for the measure.

A compact invariant set $J\subset M$ is an \emph{average conformal
repeller} if for any $\mu\in \mathcal{E}(f|_J)$, $\lambda_1(\mu)=
\lambda_2(\mu)= \cdots
 =\lambda_m(\mu)>0$.
For simplicity we denote by $\lambda(\mu)$ the unique Lyapunov
exponent with respect to $\mu$.

\begin{remark}\label{acr1}
If a compact $f$-invariant set $J$ is an average conformal
repeller, it is indeed a repeller in the usual way (\cite{c3}),
that is, $f$ is uniformly expanding on $J$.

On the other hand, there are average conformal repellers which are
not conformal  repellers (see an example in \cite{zcb}).
\end{remark}

\begin{remark}\label{acr2}
The notion of average conformal repellers is a generalization of
the quasi-conformal and asymptotically conformal repellers in
\cite{ba,pes1}. In \cite{ban}, the authors studied the dimensions
of average conformal repellers  by using thermodynamic formalism.
\end{remark}

Given a set $Z\subset M$, its {\it Hausdorff dimension} is defined by
$${\dim}_\mathrm{H}(Z)=\inf\{s: \;
\lim_{\epsilon\to 0}\inf_{\diams{\mathcal U} <\epsilon}
\sum_{U\in{\mathcal U}}(\diam U)^s=0\},$$
where ${\mathcal U}$ is a cover of $Z$ and
$\diam{\mathcal U}=\sup\{\diam U:\ U\in{\mathcal U}\}$.
If $\nu$ is a probability measure on $M$, then the Hausdorff dimension
of the measure $\nu$ is given by
$${\dim}_\mathrm{H}(\nu)
=\inf\bigl\{{\dim}_\mathrm{H}(Z): \;Z\subset M,\; \nu Z=1\;\bigr\}.$$

The \emph{upper} and \emph{lower box dimensions} of $Z$ are defined by
$$\overline{\dim}_\mathrm{B}(Z)
=\limsup_{\epsilon\to 0}\frac {\log N(\epsilon)}
{-\log \epsilon}\qquad\hbox{and}\qquad
\underline{\dim}_\mathrm{B}(Z)
=\liminf_{\epsilon\to 0}\frac {\log N(\epsilon)}
{-\log \epsilon}$$
respectively, where $N(\epsilon)$ denotes the minimum number of balls
of radius $\epsilon$ which cover $Z$.
$\overline{\dim}_\mathrm{B}(Z)$ and $\underline{\dim}_\mathrm{B}(Z)$
are also called \emph{upper} and \emph{lower capacities}.
If $\overline{\dim}_\mathrm{B}(Z)=\underline{\dim}_\mathrm{B}(Z)$,
then we simply call this number the \emph{box dimension}, and denoted it
by $\dim_\mathrm{B}(Z)$.
Similarly, for any probability measure $\nu$, we have
\begin{equation*}
\begin{split}
\overline{\dim}_\mathrm{B}(\nu)
&=\lim_{\delta\to 0} \inf\{\overline{\dim}_\mathrm{B}(Z):
Z\subset M, \ \nu Z > 1-\delta\},  \\
\underline{\dim}_\mathrm{B}(\nu)
&=\lim_{\delta\to 0} \inf\{\underline{\dim}_\mathrm{B}(Z):
Z\subset M, \ \nu Z > 1-\delta\}.
\end{split}
\end{equation*}

The \emph{upper} and \emph{lower Ledrappier dimensions}
(\cite{ledra,pes1}) is given by
$$\overline {\dim}_\mathrm{L}(\nu)
=\lim_{\delta\to 0}\limsup_{\epsilon\to 0} \frac {\log N(\epsilon,
\delta)}{-\log \epsilon}\qquad\hbox{and}\qquad \underline
{\dim}_\mathrm{L}(\nu) =\lim_{\delta\to 0}\liminf_{\epsilon\to 0}
\frac {\log N(\epsilon, \delta)}{-\log \epsilon},$$ where
$N(\epsilon, \delta)$ is the minimum number of balls of diameter
$\epsilon$ covering a subset in $M$ of measure greater than
$1-\delta$.

It is known that if a compact invariant set $J\subset M$ is a conformal
repeller for a $C^{1}$ map, then
%the Hausdorff and box dimension
%of the set can be given by the Bowen's formula, that is, let
%$\phi(x)=\log\|Df(x)\|$ be the potential, these dimensions are give by
%the unique root $s=s_0$ of the equation $P(f, -s\phi)=0$
%(see \cite{bo1}, where $P$ is the topological pressure.
%In this case,
$J$ support an ergodic measure $\mu$ such that the Hausdorff or
box dimension of measure $\mu$ are equal to the Hausdorff or box
dimension of the set $J$, which are known as the variational
principle for dimension. And these dimensions can be given by the
ratio of the measure-theoretic entropy and the Lyapunov exponent
(see e.g. \cite{bo1, rue2, gy}). See also \cite{ba2,lu} for some
recent progresses of variational principle of dimension for
non-conformal maps.

Our next theorem is a generalization of the results to
average conformal repeller of a $C^1$ map $f$.

\begin{TheoremE}
Suppose $J$ is an average conformal repeller of a $C^1$ map $f$, then
there exists an $f$-invariant ergodic measure $\mu$ supported on $J$
such that
$$D(J)=D(\mu)=\frac{h_{\mu}(f)}{\lambda(\mu)},$$
where $D(\mu)$ is $\dim_{\mathrm{H}}\mu$,
$\overline{\dim}_{\mathrm{L}}\mu$,
$\underline{\dim}_{\mathrm{L}}\mu$, $\overline{\dim}_{\mathrm{B}}\mu$
or $\underline{\dim}_{\mathrm{B}}\mu$,  and $D(J)$ is
$\dim_{\mathrm{H}}J$, $\underline{\dim}_{\mathrm{B}}J$ or
$\overline{\dim}_{\mathrm{B}}J$.
 \end{TheoremE}

In the following, for each $f$-invariant ergodic measure $\mu$
supported on $J$, we will construct a certain set $K$
whose dimension is equal to the dimension of the measure.

We denote by $m(Df(x))$ the minimal norm of $Df(x)$, that is,
$m(Df(x))=\min\{\|Df(x)u\|: u\in T_xM, \|u\|=1\}$.

\begin{TheoremF}
Suppose $J$ is an average conformal repeller of a $C^1$ map $f$,
and $\mu\in \mathcal{E}(f|_J)$.  Let
${\mathcal{F}}=\{ \log m(Df^n(x))\}_{n\geq  1}$ and
\[
K=\bigl\{x\in M:
\lim_{\epsilon\rightarrow0}\limsup_{n\rightarrow\infty}
\frac{-\log\mu(B_n(x,\epsilon))}{n}=h_\mu(T)
\  \text{and}\ \lim_{n\rightarrow\infty}\frac{1}{n} \log
m(Df^n(x))={\mathcal{F}}_*(\mu) \bigr\}.
\]
Then we have
\[
D(K)=D(\mu),
\]
where $D(\mu)$ is $\dim_{\mathrm{H}}\mu$,
$\overline{\dim}_{\mathrm{L}}\mu$,
$\underline{\dim}_{\mathrm{L}}\mu$, $\overline{\dim}_{\mathrm{B}}\mu$
or $\underline{\dim}_{\mathrm{B}}\mu$,  and $D(K)$ is
$\dim_{\mathrm{H}}K$, $\underline{\dim}_{\mathrm{B}}K$ or
$\overline{\dim}_{\mathrm{B}}K$.
\end{TheoremF}

We end the paper by provide a proof of these two theorems.

\begin{proof}[Proof of Theorem E]
In \cite{ban}, the authors proved that $D(J)=s_0$,
where $s_0$ is the unique root of the equation $P(f,-s{\mathcal{F}})=0$,
where ${\mathcal{F}}=\{ \log m(Df^n(x))\}_{n\geq  1}$.
Next, we show that the subadditive topological pressure
$P(f,-s{\mathcal{F}})$ can be attained  by an ergodic measure.

By Remark~\ref{acr1}, we know that the entropy map $\nu\mapsto
h_{\nu}(f)$ is upper-semicontinuous on $\mathcal{M}(f|_J)$. It is
also easy to check that $\nu\mapsto {-s\mathcal{F}}_{*}(\nu)$ is
upper-semicontinuous on $\mathcal{M}(f|_J)$. Hence by compactness
of $\mathcal{M}(f|_J)$, and the variational principle for
subadditive topological pressure \cite{c}, we have
$P(f,-s{\mathcal{F}})=h_{\mu_0}(f)-s{\mathcal{F}}_{*}(\mu_0)$ for
some invariant measure $\mu_0\in \mathcal{M}(f|_J)$. By ergodic
decomposition theorem (see e.g. \cite{w2}), we can obtain that
$P(f,-s{\mathcal{F}})=h_{\mu}(f)-s{\mathcal{F}}_{*}(\mu)$ for some
$f$-invariant ergodic measure $\mu$.

Note that ${\mathcal{F}}_{*}(\mu)=\lambda (\mu)$ is the unique
Lyapunov exponent with respect to $\mu$.  Hence
$D(J)=s_0=\frac{h_{\mu}(f)}{\lambda  (\mu)}$.
On the other hand, by (\cite[Corollary 2]{cao}), we can get
$D(\mu)=\frac{h_{\mu}(f)}{\lambda  (\mu)}$.
Thus we obtain the equality of the theorem.
\end{proof}

\begin{proof}[Proof of Theorem F]
For an $f$-invariant ergodic measure $\mu$ and for each $s>0$,
by Theorem 4.2 in \cite{ban}, the set $K$ is equal to the set
\[
\bigl\{x\in
M:\lim_{\epsilon\rightarrow0}\limsup_{n\rightarrow\infty}
\frac{-\log\mu(B_n(x,\epsilon))}{n}=h_\mu(T)
\ \text{and}\ \lim_{n\rightarrow\infty}\frac{-s}{n}\sum_{i=0}^{n-1}\phi
(f^ix) =-s\int \phi \mathrm{d}\mu \bigr\},
\]
where $\phi(x)=\frac 1 m\log |\det (Df(x))|$.
Using Theorem~B, we have
\begin{eqnarray}\label{equality53}
P_K(T,-s\phi)=\overline{CP}_{K}(T,-s\phi)=h_{\mu}(f)-s\int \phi
\mathrm{d}\mu.
\end{eqnarray}
For every subset $Z\subset J$, it is proved in \cite{cao} that
\[
\dim_\mathrm{H} Z\geq t^* \ \ \text{and} \ \
\overline{\dim}_{\mathrm{B}}Z\leq s^*,
\]
where $t^*$ and $s^*$ are the unique root of the Bowen's equation
$P_Z(f,-t\phi)=0$ and $\overline{CP}_Z(f, -s\phi)=0$ respectively.
Using this fact and (\ref{equality53}), we have
\[
\dim_{\mathrm{H}}K\geq
\frac{h_{\mu}(f)}{{\mathcal{F}}_*(\mu)} \ \ \text{and} \ \
\overline{\dim}_{\mathrm{B}}K\leq
\frac{h_{\mu}(f)}{{\mathcal{F}}_*(\mu)}
\]
since ${\mathcal{F}}_*(\mu)=\int \phi \mathrm{d}\mu$.  Note that
${\mathcal{F}}_*(\mu)$ is equal to the unique Lyapunov exponent
$\lambda(\mu)$ because $J$ is an average conformal repeller.
Combining the above inequalities and the result in Theorem~E, we have
\[
D(K)=\frac{h_{\mu}(f)}{\lambda(\mu)}=D(\mu).
\]
\end{proof}

\vskip0.2cm \noindent {\bf Acknowledgements.}  Authors would like
to thank Wen Huang and Pengfei Zhang for giving us the manuscript
\cite{huang}. They would also like to thank the referee for
providing us invaluable suggestions, which led to great
improvements of this paper.  Cao is partially supported by NSFC
(10971151,11125103) and the 973 Project (2007CB814800), Zhao is
partially supported by NSFC (11001191), Ph.D. Programs Foundation
of Ministry of Education of China (20103201120001) and NSF in
Jiangsu province (09KJB110007).


\begin{thebibliography}{999}


\bibitem{ban}
J. Ban, Y. Cao and H. Hu,
Dimensions of average conformal repeller,
{\it Trans. Amer. Math. Soc.}, {\bf 362} (2010), 727-751.

\bibitem{ba}  L. Barreira,
A non-additive thermodynamic formalism and applications
to dimension theory of hyperbolic dynamical systems,
{\it Ergodic Theory and Dynamical Systems}, {\bf 16} (1996), 871-927.

\bibitem{ba2} L. Barreira, C. Wolf,
 Measures of maximal dimension for hyperbolic diffeomorphisms,
{\it Commun. Math. Phys.}, {\bf 239} (2003), 93-113.




\bibitem{bo1} R. Bowen,
Topological entropy for noncompact sets,
{\it Trans. Amer. Math. Soc.}, {\bf 184} (1973), 125-136.

\bibitem{bo3} R. Bowen,
Equilibrium states and the ergodic theory of Anosov diffeomorphisms,
{\it Lecture notes in Math.}, 470, Springer-Verlag, 1975.

\bibitem{bo2} R. Bowen,
Hausdorff dimension of quasicircles,
{\it Inst. Haustes \'Etudes Sci. Publ. Math.}, {\bf 50} (1979), 11-25.

\bibitem{brin} M. Brin and A. Katok,
On local entropy,
{\it Lecture Notes in Mathematics}, 1007, Springer-Verlag, 1983.

\bibitem{c} Y. Cao, D. Feng and W. Huang,
The thermodynamic formalism for sub-multiplicative potentials,
{\it Discrete Contin. Dynam. Syst. Ser. A}, {\bf 20} (2008), 639-657.


\bibitem{c3} Y. Cao,
Nonzero Lyapunov exponents and uniform hyperbolicity,
{\it Nonlinearity}, {\bf 16} (2003), 1473-1479.

\bibitem{cao} Y. Cao,
Dimension estimate and multifractal analysis of $C^1$
average conformal repeller,
{\it Preprint}.

\bibitem{czc} W. Cheng, Y. Zhao and Y. Cao,
Pressures for asymptotically subadditive potentials under a
mistake funciton, {\it Discrete Contin. Dynam. Syst. Ser. A}, {\bf
32} (2012), 487-497.

\bibitem{cli} V. Climenhaga,
Bowen's equation in the non-uniform setting,
{\it Ergodic Theory and Dynamical Systems}, {\bf 31} 2011, 1163-1182.


\bibitem{falconer} K. Falconer,
A subadditive thermodynamic formalsim for mixing repellers,
{\it J. Phys. A}, {\bf 21} (1988), 737-742.

\bibitem{gy} D. Gatzouras and Y. peres,
The variational principle for Hausdorff dimension: a survay,
{\it London Math. Soc. Lecture Note Series} 228,
Cambridge University Press, (1996), 113-126.

\bibitem{h} L. He, J. Lv and L. Zhou,
Definition of measure-theoretic pressure using spanning sets,
{\it Acta Math. Sinica(English Series)}, {\bf 20} (2004), 709-718.



\bibitem{huang} W. Huang and P. Zhang,
Pointwise dimension, entropy and Lyapunov exponents for $C^1$ map,
To appear in {\it Trans. Amer. Math. Soc.}

\bibitem{kat} A. Katok,
Lyapunov exponents, entropy and periodic orbits for diffeomorphisms,
{\it Inst. Hautes \'Etudes Sci. Publ. Math.}, {\bf 51} (1980), 137-173.

\bibitem{king} J. F.C. Kingman,
The ergodic theory of subadditive stochastic processes,
{\it J. Royal Stat. Soc. Ser B}, {\bf 30} (1968), 499-510.

\bibitem{ledra} F. Ledrappier,
Some relations between dimension and Lyapunov exponent,
{\it Commun. Math. Phys.}, {\bf 81} (1981), 229-238.

\bibitem{lu} N. Luzia,
A variational principle for the dimension for a class of
non-conformal repellers,
{\it Ergodic Theory and Dynamical Systems}, {\bf 26} (2006), 821-845.

\bibitem{osel} V. Oseledec,
A multiplicative ergodic theorem,
characteristic Lyapnov exponents of dynamical systems,
{\it Transactions of the Moscow Mathematical Society}, {\bf 19} (1968),
American Mathematical Society, Providence, R.I..

\bibitem{pes1} Ya. Pesin,
Dimension theory in dynamical systems, Contemporary
Views and Applications,
{\it University of  Chicago  Press}, Chicago, 1997.

\bibitem{pes2} Ya. Pesin and B. Pitskel',
Topological pressure and the variational
principle for noncompact sets,
{\it Functional Anal. Appl.}, {\bf 18} (1984), 307-318.

\bibitem{rue} D. Ruelle,
Statistical mechanics on a compact set with $Z^{\nu}$
action satisfying expansiveness and specification,
{\it Trans. Amer. Math. Soc.}, {\bf 187} (1973), 237-251.

\bibitem{rue2} D. Ruelle,
Repellers for real analytic maps,
{\it  Ergodic Theory Dynamical Systems}, {\bf 2} (1982), 99-107.

\bibitem{wal} P. Walters,
A variational principle for the pressure of continuous transformations,
{\it Amer. J. Math.}, {\bf 97} (1975), 937-971.

\bibitem{w2} P. Walters,
An introduction to ergodic theory,
{\it Springer-Verlag}, New York, 1982.



\bibitem{gh} G. Zhang,
Variational principles of pressure,
{\it Discrete Continuous Dynam. Systems A}, {\bf 24} (2009), 1409-1435.

\bibitem{zhang} Y. Zhang,
Dynamical upper bounds for Hausdorff  dimension of invariant sets,
{\it Ergodic Theory and Dynamical Systems}, {\bf 17} (1997), 739-756.

\bibitem{z1} Y. Zhao and  Y. Cao,
Measure-theoretic pressure for subadditive potentials,
{\it Nonlinear analysis}, {\bf 70} (2009),  2237-2247.

\bibitem{z2} Y. Zhao,
A note on the measure-theoretic pressure in subadditive case,
{\it Chinese Annals of Math., Series A}, {\bf No. 3} (2008), 325-332.

\bibitem{z3} Y. Zhao and Y. Cao,
On the topological pressure of random bundle transformations
in subadditive case,
{\it J. Math. Anal. Appl.}, {\bf 342} (2008), 715-725.

\bibitem{zcb} Y. Zhao, Y. Cao and J. Ban,
The Hausdorff dimension of average conformal repellers
under random perturbation,
{\it Nonlinearity}, {\bf 22} (2009), 2405-2416.

\end{thebibliography}
\end{document}